\pdfoutput=1
\documentclass[12pt]{article}

\usepackage[T1]{fontenc}
\usepackage{graphicx}
\usepackage[numbers]{natbib}
\usepackage[font=footnotesize]{caption}

\usepackage{amssymb}
\usepackage{xcolor}
\usepackage{amsmath}
\usepackage{amsthm}
\newtheorem{theorem}{Theorem}[section]
\newtheorem{lemma}[theorem]{Lemma}

\newtheorem{remark}{Remark}
\newtheorem{assumption}{Assumption}[section]

\usepackage{hyperref}
\hypersetup{
  pdftitle={Asymptotic approximation of sensitivities in finite dimensional continuous data assimilation with application to parameter estimation},
  pdfauthor={Joshua Newey, Jared P. Whitehead}
}

\makeatletter
\renewcommand\section{\@startsection{section}{1}{\z@}%
  {-3.5ex \@plus -1ex \@minus -.2ex}%
  {2.3ex \@plus.2ex}%
  {\normalfont\Large\bfseries\raggedright}}
\renewcommand\subsection{\@startsection{subsection}{2}{\z@}%
  {-3.25ex\@plus -1ex \@minus -.2ex}%
  {1.5ex \@plus .2ex}%
  {\normalfont\large\bfseries\raggedright}}
\makeatother

\title{Asymptotic approximation of sensitivities in finite dimensional continuous data assimilation with application\\ to parameter estimation}

\author{Joshua Newey\\
  \small University of California Davis\\
  \small One Shields Ave\\
  \small Davis, CA 95616, USA
  \and
  Jared P. Whitehead\\
  \small Brigham Young University\\
  \small 275 TMCB\\
  \small Provo, UT 84602, USA}

\date{}

\begin{document}

\maketitle

\begin{abstract}
We develop a rigorous justification for a parameter estimation algorithm which couples continuous data assimilation to generic optimization.  For finite dimensional systems, we provide a rigorous justification of an asymptotic approximation of the sensitivity for the underlying modeled dynamical system, prove that the $L^2$ loss function satisfies an approximate Polyak-Lojasiewicz inequality, and use that result to justify convergence of gradient descent for the proposed algorithm.  Numerical examples are provided that demonstrate the precision of the rigorous results.
\end{abstract}

%%%%%%%%%%%%%%%%%%%%%%%%%%%%%%%%%%%%%%%%%%%%%%%%%%%%%%%%%%%%%%
%*************************************************************
%%%%%%%%%%%%%%%%%%%%%%%%%%%%%%%%%%%%%%%%%%%%%%%%%%%%%%%%%%%%%%

\section{Introduction}
For several decades the standard approach in scientific computation and simulation has been to start with a dynamical set of equations for the system to be modeled.  This dynamical system is then discretized in some fashion, and numerical algorithms are employed to approximate solutions which are meant to imitate reality.  Some of the issues with this approach come down to: 1) The lack of a perfectly observed initial condition for most systems. 2) Restrictions on numerical resolution that do not allow the complete resolution of all scales (in both time and space) of interest. And 3) The potential for missing physics or processes in the model itself.  In recent years these issues have been partially addressed by merging the traditional derivation of dynamical systems with data driven techniques.

Data assimilation in its many forms \cite{law2015data,asch2016data,carrassi2018data} addresses the first issue by incorporating observational data as it becomes available and augmenting the simulated state with the observed values, avoiding the necessity to exactly specify the initial condition for the simulation.  Sub-grid scale modeling and parameterizations \cite{alizadeh2022advances} address the second issue by replacing the unresolved scales with simplified models often derived from statistical assumptions on the system in question (see \cite{beck2019deep,grooms2023backscatter,guan2023learning} for example).  Parameterization of the unresolved scales is closely related to the third issue: identification of terms in the equations that are physically relevant yet unknown a priori (see \cite{brunton2016discovering,mangan2017model,wu2019numerical,cranmer2020discovering,messenger2021weak,viknesh2026adam,alcalde2026data} for example).  The overlap of these three issues is the setting for the rigorous justification provided by the present article: we present a theoretical foundation for an algorithm under which the model state is synchronized with observations while the unknown portion of the model is inferred from the same partially observed dynamics.

Continuous data assimilation is generically the utilization of continuous in time partial observations of a system, coupled with a sophisticated physical model, to ascertain the full state of the system.  Continuous data assimilation in concert with parameter recovery has developed in several complementary directions.  Early work used the persistent synchronization error to dynamically learn unknown parameters in dissipative systems \cite{CHL_2020,Lorenz_2021,Martinez_2022}, while more recent work has interpreted these updates as finite-dimensional root-finding and optimization algorithms \cite{newey2025model,martinez2025relaxation}.  Related ideas have also been used to identify more general model components, including unknown body forces in the two-dimensional incompressible Navier--Stokes equations from partial velocity observations \cite{farhat2024identifying,martinez2025reconstruction}.  These results suggest that the nudging mechanism is useful not only as a state-estimation procedure, but also as a way of producing computationally accessible inverse problems for unknown parameters or missing terms.

Several very recent studies further illustrate the current scope of this approach.  Ahmad, Biswas, and Hoffman \cite{ahmad2026dataassimilationaugmented} formulate a data-assimilation-augmented optimization framework for ODE parameter estimation from partial observations, producing a very comprehensive analysis of the identifiability and sensitivity for a time-delayed mismatch integrated loss functional for the Lorenz '63 dynamical system.  Wang, Carlson, and Hoffmann \cite{wang2026comparing} compare deterministic (including the algorithm discussed in the current paper) and stochastic parameter recovery algorithms on noisy Lorenz '63 and multiscale Lorenz '96 data, and report that deterministic parameter recovery paired with deterministic data assimilation can be more accurate and computationally efficient in their tests.  Li and Sanz-Alonso \cite{li2026continuoussurrogate} study continuous data assimilation with learned surrogate dynamics, proving that nudging with an imperfect learned model can retain exponential tracking up to an error floor determined by observation noise and surrogate approximation error.  The present work is complementary to these developments: rather than proposing a new cost functional, comparing stochastic and deterministic implementations, or replacing the dynamics by a learned surrogate, we isolate and justify the large-$\mu$ (nudging coefficient strength) asymptotic sensitivity approximation that underlies the efficient algorithms proposed in \cite{newey2025model}.  In addition, we utilize the large-$\mu$ assumption to justify convergence of the underlying optimization procedure, demonstrating that the specified loss function satisfies the Polyak-Lojasiewicz inequality which allows for rigorous bounds on the convergence of gradient descent or similar gradient based optimization algorithms.

\subsection{Basic Data Assimilation Problem Setup}
Consider the nudged, coupled system
\begin{align} \label{eq:truth}
    \frac{\partial u}{\partial t} &+ F(u;a)  = 0 \\ \label{eq:nudge}
    \frac{\partial v}{\partial t} &+ F(v;b)
    + \mu I_h(v-u) = 0,
\end{align}
with initial conditions $u(0) = u_0, v(0) = v_0 \neq u_0$. 
Precise assumptions on $F,~\mu$ and $I_h$ will be outlined below, but generically we assume that the system is such that $v(t) \rightarrow u(t)$ as $t\rightarrow \infty$ when $a=b$ (see \cite{AOT_2014} for the first analysis of this setup for the 2D Navier-Stokes equations).

Since the original presentation in \cite{AOT_2014}, several significant advancements have been made which extend this approach to additional dynamical settings with relaxed restrictions on the observations (when $a=b$). We will not review all of the relevant results here, but will highlight a few key contributions.  Convergence and error estimates have been extended from exact data available continuously in time at a fixed collection of coarse degrees of freedom, to observations corrupted by stochastic noise \cite{CDA_2_noise_2015}, supplied only at discrete times \cite{foias2016discrete} (see also the computational study \cite{hammoud2023continuous}), averaged over short temporal windows \cite{jolly2019continuous}, or placed on irregular, mesh-free collections of sensors \cite{biswas2021mesh}.  Computational studies with moving observers further indicate that the geometry and motion of the observation network can change the convergence rate by up to an order of magnitude and can permit accurate reconstruction with substantially fewer observational sensors \cite{franz2022dynamic}.

Extensions beyond 2D Navier-Stokes have also been demonstrated both rigorously and numerically.  Conditions stated directly in terms of the observed data have been developed for well-posedness and synchronization in the three-dimensional Navier--Stokes equations \cite{biswas2021three}, and AOT has been implemented in high-resolution global ocean calculations with realistic coastlines and bathymetry, where the computations reveal both rapid initial error reduction and model-dependent asymptotic error floors \cite{larios2024ocean}.  This method has been used for dynamical downscaling in a regional climate model \cite{desamsetti2022enhanced}, and a physics-informed neural network surrogate of such downscaling has been analyzed numerically for Rayleigh--B\'enard convection under observational and model noise \cite{hammoud2025downscaling} (see also \cite{hammoud2024data} for a deep reinforcement learning approach to data assimilation in chaotic systems).  In parallel, adaptive choices of the nudging coefficient have begun to replace conservative a priori bounds by gains selected from the evolving flow and assimilation error \cite{cibik2025adaptive}.  Thus, the observation resolution and the relaxation parameter are increasingly treated as coupled design variables that must balance synchronization, observational uncertainty, model error, and computational practicality.  This point is particularly relevant in the current study, as we rely on large $\mu$.

More recent studies have compared AOT directly with ensemble Kalman filtering in the presence of measurement error \cite{ning2024assessment}, extended the method to non-interpolant observables and problems with limited regularity \cite{yushutin2025continuous}, and established stochastic analogues for the two-dimensional Navier--Stokes equations \cite{bessaih2025stochastic}.  At the same time, \cite{titi2025inadequacy} demonstrates that dissipativity and determining observations are not merely technical conveniences: for non-dissipative systems, sparse observations may fail to distinguish among infinitely many trajectories.  These recent results motivate the explicit observability and dissipativity assumptions imposed below.

\subsection{Parameter estimation framework}
In our discussion and derivation below, the system for $u$ \eqref{eq:truth} is referred to as the `reference' or `drive' system.
The system for $v$ \eqref{eq:nudge} is referred to as the `model' or `response' system. The operator $I_h$ is the observation operator so that $I_hu$ denotes the components of the reference solution that are observable at each time $t$.  Along those lines, note that we are assuming that these observations are available continuously in time.  If that is not the case, an interpolant in time can be utilized, at the cost of additional restrictions on the time between observations (see \cite{foias2016discrete} for example).

Note that for this paper we will consider the case when both systems have the same functional form, i.e. $F$ is the same in both equations (this simplifies much of the analysis and notation below).  
This allows us to make more rigorous statements about convergence of the parameter estimation algorithm.  
This restriction is not fundamental to the algorithm however, and the rigorous results described here could be extended to a more general setting, at the cost of readability of the current manuscript.

Following \cite{newey2025model}, we seek the parameter $b$ which minimizes the error function 
\begin{align}\label{eq:error_func}
    E(b) = \left.\frac{1}{2}\| I_hv(b) - I_hu \|^2 \right|_{t=T},
\end{align}
where $T$ is some large value (see \cite{ahmad2026dataassimilationaugmented,schill2026asymptotic} for a time-averaged version of this same cost functional).
As noted above, in the precise case where the functional form of the reference and response system are identical, then the optimal solution will be $b\rightarrow a$ which is satisfied in the long time limit as $\lim_{T\rightarrow \infty} E(a) = 0$ under the appropriate assumptions on the system setup.  In \cite{newey2025model} it is noted that gradient based optimization and root-finding algorithms can be applied to \eqref{eq:error_func} utilizing the sensitivity equations of \eqref{eq:nudge}, but the key numerical observation there is that large $\mu$ asymptotics on said sensitivity equations leads to a computationally efficient algorithm that appears numerically to work in a variety of settings.  The current article provides rigorous justification for that approach.  This is completed below where Section \ref{sec:assumptions} describes the setup of the problem and the requisite assumptions on the dynamical system, Section \ref{sec:nudging_convergence} provides a rigorous justification for the convergence of \eqref{eq:nudge} to \eqref{eq:truth} when $a=b$ (see \cite{li2026continuoussurrogate} for a slightly different approach) in finite dimensions.  Section \ref{sec:rigorous_asymptotics} provides a rigorous justification for the asymptotic approximation motivated in \cite{newey2025model}, Section \ref{sec:rigorous_optimize} justifies convergence of gradient based optimization algorithms using the approximate, asymptotic gradient including that \eqref{eq:error_func} satisfies the PL-inequality.  Finally Section \ref{sec:examples} details how the underlying assumptions are satisfied in a simple setting and Section \ref{sec:conclusions} concludes with some discussion on future work and possible challenges for both the practical method and its rigorous analysis.

\section{Summary of the general algorithm and key assumptions}\label{sec:assumptions}
\subsection{The basic algorithm}
Under suitable assumptions which are described in detail below, identification of the optimal parameter minimizing \eqref{eq:error_func} can be performed via any gradient based optimization routine (see \cite{sun2020optimization} for a partial review of such algorithms).  For example, gradient descent leads to the following iterative scheme:
\begin{equation}\label{eq:gradient_descent}
    b_{k+1} = b_k - \gamma_k 
    \nabla E(b_k),
\end{equation}
where the learning rate $\gamma_k$ is selected to ensure a desired convergence (typically the learning rate $\gamma_k$ should decay with iteration $k$), and we assume that the initial guess $b_0$ is sufficiently close to the desired minimum.  The difficulty of course is that calculation of $\nabla E(b_k)$ is expensive and/or impractical without significantly rewriting an existing codebase.  For instance, in the specific case where the unknown parameter $b$ is a scalar then this gradient becomes the derivative with respect to the parameter:
\begin{equation}
    \left.\frac{\partial E}{\partial b} = \left\langle \frac{\partial}{\partial b} I_h v(b),I_hv(b)-I_h u\right\rangle\right|_{t=T},
\end{equation}
where we have assumed at least for now that the norm used in \eqref{eq:error_func} is the $L^2$ norm with accompanying inner product. We still must compute $\frac{\partial v}{\partial b}$ at time $T$. This can of course be accomplished by implicitly differentiating \eqref{eq:nudge} with respect to the parameter $b$ which identifies the sensitivity equation for $W=D_b v$ as
\begin{equation}\label{eq:sensitivity}
    \frac{\partial W}{\partial t} + D_v F(v;b)W + D_bF(v;b) +  \mu I_h W = 0,
\end{equation}
where $D_vF$ is the Jacobian of $F$ with respect to $v$ and $D_bF$ is the gradient when $b\in\mathbb{R}$ and the Jacobian with respect to the vector $b$ when there are multiple parameters (we will maintain the same notation in either case).  
Of course this requires the solution of an additional differential equation, in addition to \eqref{eq:nudge}.  For a single parameter $b$ this will not overly increase the computational cost, but for multiple parameters (vector valued $b$) this becomes computationally prohibitive, an issue that may be partially circumvented via an adjoint approach
(see \cite{schill2026asymptotic} for a presentation of just such an adjoint-based approach via modifying the loss function \eqref{eq:error_func} to be a time-averaged loss), but the additional cost in any case is significant.

The insight provided in \cite{newey2025model} is that in the context of data assimilation, these systems are considered for large values of $\mu$.  As such, the observable part of $W$ is approximated asymptotically as
\begin{equation}\label{eq:asymptotic}
    I_hW \sim -\frac{1}{\mu}I_h D_bF ,
\end{equation}
so long as the parameter $b$ only appears in the observable part of the original system \eqref{eq:nudge}.  
Using this asymptotic approximation to $I_hW$ provides a computationally viable alternative to solving \eqref{eq:sensitivity} that numerically appears to work remarkably well in a variety of settings as indicated in \cite{newey2025model}.

In summary then, the gradient descent algorithm becomes
\begin{align}
    b_{k+1} &= b_k + \frac{\gamma_k}{\mu} \left.\left\langle 
    D_{b_k} F
    ,I_hv(t;b_k) - I_hu(t)\right\rangle\right|_{t=T},
\end{align}
where for a single parameter $\frac{\partial F}{\partial b_k}  = D_{b_k}F$.
The purpose of this article is to clarify the required assumptions \eqref{eq:truth} and \eqref{eq:nudge} must fulfill as well as to rigorously justify the asymptotic approximation \eqref{eq:asymptotic} and its subsequent use in gradient-based optimization routines.

\begin{remark}
When the parameter $b$ is a vector of unknown quantities then the sensitivity will be  a matrix given by 
\begin{equation}\label{eq:sensitivities_vector}
    W(t;b) = D_bv(t;b),
\end{equation}
where $D_b$ is the Jacobian with respect to parameters.
Following this, we will refer to the $W$ as the sensitivity operator.
\end{remark}

\subsection{Key assumptions}
Here we will state the key assumptions required for the rigorous results produced below. 
In each case we will also clarify which assumptions are critical to the final results, and highlight those that are only necessary to simplify the resultant analysis, but are not essential to the final result.  
Rather than simply listing each assumption, we will also provide commentary that clarifies the need for each assumption. A similar approach is taken in \cite{li2026continuoussurrogate} where synchronization is justified for finite dimensional systems when the underlying model is only a surrogate and not the full, true dynamical system.

 We will suppose that the vector of possible parameters $b \in B \subset \mathbb{R}^m$ where $B$ is referred to as the admissible set of parameters.  
 Furthermore, $u(t),v(t):\mathbb{R} \rightarrow \mathbb{R}^n$ and we will further suppose that $I_h:\mathbb{R}^n\rightarrow\mathbb{R}^n$ with $\text{dim } \text{range } I_h = q\leq n$.  
 This implies that $F(v;b):\mathbb{R}^n\times\mathbb{R}^m \rightarrow \mathbb{R}^n$.  
 It also follows that the sensitivity is a matrix  $W\in\mathbb{R}^{n\times m}$.
Following this discussion, we will refer to $\mathbb{R}^n$ as \emph{state space}, and $\mathbb{R}^m$ as \emph{parameter space}.

In the following, $\|x\|$ and $\langle x, y \rangle$  refers respectively to the $l^2$ norm and corresponding inner product on $\mathbb{R}^n$.
We will also use the notation $|b|$ and $a \cdot b$ to denote the $l^2$ norm and corresponding inner product of vector valued parameters $b$.  
We will also use the notation $\|. \|$ to refer to the induced spectral norm on finite dimensional linear operators.

\begin{assumption}\label{assumption:projection_obs}
    The observation operator $I_h$ is a linear projection operator (i.e. idempotent, hermitian) with orthogonal complement $J_h$.
\end{assumption}
The linearity of the observation operator is most certainly critical for everything that follows, and we do not consider nonlinear observations here (see \cite{ahmad2026dataassimilationaugmented} for synchronization and identifiability analyses with a linear observation operator).  The restriction to a projection with the orthogonal complement $J_h$ is primarily to simplify the presentation that follows (see  \cite{foias2016discrete,farhat2018assimilation,jolly2019continuous,biswas2021data,biswas2021mesh,yushutin2025continuous} for a few examples of other types of observation operators in the continuous data assimilation setting).

\begin{assumption}\label{assumption:observable_parameters}
The unidentified parameters $b$ are restricted to the observable portion of the dynamics, i.e.
\begin{align}
    J_h D_bF = 0
\end{align}
\end{assumption}
This condition is an essential yet natural restriction that requires that the sought after parameters interact with portions of the state that are observable via $I_h$.  In other words, we can only recover/identify unknown parameters via the asymptotic reduction justified below for parameters that directly influence the observed portion of the state.  This particular restriction is only necessary for the asymptotic representation of the gradient \eqref{eq:asymptotic} and isn't necessary for the optimization problem itself, only for this approximate gradient.

\begin{assumption}\label{assumption:absorbing_ball}
    Solutions to the coupled system \eqref{eq:truth} and \eqref{eq:nudge} have a forward invariant global absorbing set $A_{b,\mu}$ in the cross product space for each value of $b\in B$ the admissible set of parameters,
    and there exists a ball $\mathcal{B}$ (centered at $0$) of radius $R$ such that
    \begin{equation}
        \bigcup_{b\in B, \mu \geq\mu_0} A_{b,\mu} \subset \mathcal{B},
    \end{equation}
    where $\mu_0$ is described in assumption \ref{assumption:jacobian_dissipation} below.
    We will refer to the projection onto the $v$-subspace of $\mathcal{B}$ as $\mathcal{B}_v$.
    We will assume initial conditions for $(u(0),v(0))$ are chosen from within $A_{b,\mu}$ and thus stay within $\mathcal{B}$.
\end{assumption}
This indicates that the solutions remain bounded for all time and that the dissipation in the system is sufficient to control the positive Lyapunov exponents that may arise.  This particular assumption is more generic than the previous one and is primarily relevant for justifying the convergence of the nudged solution \eqref{eq:nudge} to \eqref{eq:truth} plus a quantifiable `model error' term $O(|b-a|)$.

The next assumption is a summary of technical conditions on the dynamical system itself that are necessary to guarantee convergence.  There are other versions of this and the final assumption (stated below) which would also yield the desired final results (see \cite{li2026continuoussurrogate} for a different yet very related approach).  The assumptions stated here are not universal, or the most generally stated, but do provide an avenue to produce a readable and self-enclosed derivation.
\begin{assumption}\label{assumption:technical}
    Assume that $F(v;b)$ is continuously differentiable in $v$ for all $v$ inside the absorbing ball implying that $F$ is Lipschitz (uniformly over $b$) in $v$ with Lipschitz constant
\begin{align}
    L_F = \sup_{v \in \mathcal{B}_v,b \in B } \left\|
    D_vF(v;b)\right\|
    \label{eq:Lipschitz}
    .
\end{align}
We also require $F$ to be Lipschitz in the parameter $b$
\begin{align}
    \| F(v;b) - F(v;a) \|
    \leq C_F |b-a|.
\end{align}
Finally, we also require that the mixed second derivative of $F$ is also bounded, i.e.
\begin{equation}
    \left\|D_b D_vF(v;b)\right\| \leq L_{Fb}.
\end{equation}
$D_b D_v F(v;b)$ is a linear map from $\mathbb{R}^n\times \mathbb{R}^m \to \mathbb{R}^n $ and the norm used here is the operator norm.
All the constants above ($L_F,~C_F,~L_{Fb}$) are finite.
\end{assumption}

\begin{assumption}\label{assumption:jacobian_dissipation}
    We will assume that there exists a $\mu_0 > 0$, and  some symmetric positive definite matrix $P$.
    Such that for $\mu\geq\mu_0$ we have
    \begin{align}
        P (D_vF + \mu I_h)
        + (D_vF^T + \mu I_h) P
        \succeq \alpha I \label{eq:Dissipation},
    \end{align}
    for some $\alpha > 0$, where $I$ is the identity matrix, and for all $v \in \mathcal{B}_v$ and $b \in B$.
\end{assumption}

Notice that \eqref{eq:Dissipation} implies that there is a quadratic form $\langle P x, x \rangle$ which satisfies
\begin{align}
    \langle P ( D_vF(v;b) + \mu I_h )  x ,
    x \rangle &\geq  \frac{\alpha}{2} \| x\|^2,
\end{align}
this quadratic form gives us an equivalent norm on $\mathbb{R}^n$ given by 
\begin{align}
    \|x\|_P = \| P^{\frac{1}{2}} x\|.
\end{align}
This satisfies
\begin{align}
     \sqrt{\lambda_{\text{min}}(P)} \|x\| 
     \leq \|x\|_P \leq \sqrt{\lambda_{\text{max}}(P)} \|x\|
     . 
\end{align}

This condition guarantees the existence of a globally quadratic Lyapunov function for the linearized system
\begin{align}
    \dot{x}
    +D_vF(v;b) x + \mu I_h x= 0.
\end{align}
As shown below, this guarantees both the exponential convergence of nudging when the parameters agree and the desired asymptotic behavior of the sensitivities.
As we shall see, the necessary property for the results below is that the linearized system satisfies an exponential decay property for $\mu$ large enough.
Other conditions to guarantee this are certainly possible, however many would depend on the detailed dynamics of the system (as the linearized system depends on $v(t)$).
We make the choice here, as it is simple but sufficiently general to meet our needs.

We next give a sufficient criterion for verifying assumption \ref{assumption:jacobian_dissipation} which will be useful when we discuss specific examples.
\begin{lemma}\label{lem:jacobian_dissipation}
    Given the fact that $F$ is Lipschitz, then if we can find a symmetric positive definite matrix $P_J$ defined with range equal to the range of $J_h$ so that
    \begin{align}
        P_J J_h \{D_vF\} J_h
        + J_h \{D_vF\}^T J_h P_J \succeq \gamma J_h,
    \end{align}
    for some $\gamma > 0$, then Assumption \ref{assumption:jacobian_dissipation} holds.
\end{lemma}
\begin{proof}
    Without loss of generality, we assume that the space is partitioned so that we can define $P$ on all of $\mathbb{R}^n$ in block matrix form on the $I_h$ and $J_h$ subspaces as
    \begin{align}
        P = I_h + P_J.
    \end{align}
    Because the range of $P_J$ is the range of $J_h$, and $P$ is symmetric, we see that $P J_h = J_h P = P_J $.
    This shows us that $P I_h = I_h P = I_h$. 
    From this we see  
    \begin{align}
        \|I_h x\|_P &= \|I_h x\|,
        \notag\\
        \|J_h x\|_P &= \| J_h x\|_{P_J}.
    \end{align}
    Using this and the fact that $I= I_h + J_h$, letting $x\in\mathbb{R}^n$ be arbitrary, we can write
    \begin{align}
        \langle P \{D_vF\} x , x \rangle
        &= \langle P \{D_vF\} J_h x, J_hx \rangle
        + \langle P \{D_vF\} J_h x, I_h x \rangle
        + \langle P \{D_vF\} I_h x, J_hx \rangle\notag\\
        &+ \langle P\{D_vF\} I_h x, I_h x \rangle
        \notag\\
        &=
        \langle P_J J_h \{D_vF\} J_h x, J_hx \rangle
        + \langle \{D_vF\} J_h x, I_h x \rangle
        + \langle P \{D_vF\} I_h x, J_hx \rangle\notag\\
        &+ \langle \{D_vF\} I_h x, I_h x \rangle
    \end{align}
    By the assumptions of the lemma we have 
    \begin{align}
        \langle P_J J_h \{D_vF\} J_h x, J_hx \rangle
        \geq \frac{\gamma}{2} \| J_h x\|^2.
    \end{align}
    $F$ being Lipschitz gives us
    \begin{align}
        \langle \{D_vF\} I_h x , I_h x \rangle 
        \geq -L_F \|I_h x\|^2,
    \end{align}
    and using Cauchy Schwarz and Young's inequalities we can write
    \begin{align}
        \langle  P \{D_vF\} J_h x, I_h x \rangle
        &=
        \langle \{D_vF\} J_h x, I_h x \rangle
        \geq -L_F \|J_h x\| \| I_h x\|
        \notag\\
        &\geq - \frac{\gamma}{8} \|J_hx \|^2
        - \frac{4L_F^2}{ \gamma} \|I_h x\|^2 
        \\
        \langle P \{D_vF\} I_h x, J_h x \rangle
        &\geq -L_F \lambda_{\max}(P_J) \|J_h x\| \| I_h x\|
        \notag\\
        &\geq - \frac{\gamma}{8} \|J_hx \|^2
        - \frac{4L_F^2 \lambda_{\max}(P_J)^2}{ \gamma} \|I_h x \|^2.
    \end{align}
    Putting everything together gives us 
    \begin{align}
        \langle P \{D_vF\} x , x \rangle
        \geq \frac{\gamma}{4} \|J_h x\|^2
        - \left(
        L_F + \frac{4L_F^2 \lambda_{\max}(P_J)^2}{ \gamma}
        + \frac{4L_F^2}{ \gamma}
        \right) 
        \| I_h x\|^2.
    \end{align}
    Choosing
    \begin{align}
        \mu_0 =\left(
        L_F + \frac{4L_F^2 \lambda_{\max}(P_J)^2}{ \gamma}
        + \frac{4L_F^2}{ \gamma}
        \right) + \frac{\gamma}{4}, 
    \end{align}
    then for $\mu> \mu_0$ we have 
    \begin{align}
        \langle P  \{D_vF\} x, x \rangle
        + \mu \langle P I_h x,x \rangle
        &=
        \langle P \{D_vF\} x, x \rangle
        + \mu \| I_h x \|^2
        \notag\\
        &\geq
        \frac{\gamma}{4}\|J_h x\|^2
        + \frac{\gamma}{4}\|I_h x\|^2
        = \frac{\gamma}{4} \|x\|^2.
    \end{align}
    This concludes the proof with $\alpha = \gamma/2$.
\end{proof}

Weaker or more general conditions for dissipation are possible.  See \cite{li2026continuoussurrogate} or \cite{ahmad2026dataassimilationaugmented} for different approaches to obtain results similar to the next section.

\section{Convergence of Nudging}\label{sec:nudging_convergence}

In order for our optimization algorithm to be well posed (e.g. to have a meaningful minimum), we first need results about the convergence of nudging, i.e. when $a=b$.  This question is also addressed in \cite{li2026continuoussurrogate} in a different light but is presented here to maintain continuity in the argument.

\begin{lemma}\label{lem:bounded_error}
    For the coupled system
 \eqref{eq:truth}-\eqref{eq:nudge},    
with the nudging parameter chosen so that $\mu > \mu_0$ with $\mu_0$ chosen according to assumption \ref{assumption:jacobian_dissipation} then the state error $z=v-u$ is bounded for all time as
    \begin{equation}
        \|z(t)\|^2 \leq C_1 \|z(0)\|^2e^{-\gamma t} + \frac{C_2}{\gamma^2}|b-a|^2,
    \end{equation}
    for some positive constants $\gamma, C_1, C_2$.
\end{lemma}
\begin{proof}
    The evolution of $z = v-u$ is given by
    \begin{align}
        \frac{dz}{dt} &=    \left[F(u;a)-F(v;b)\right]  - \mu I_h z\\
        &= - \mu I_h z - \left[F(v;b)-F(u;b)\right] - \left[F(u;b)-F(u;a)\right]. \label{eq:error_evolve1}
    \end{align}
    By Assumption \ref{assumption:jacobian_dissipation} there exists some symmetric positive definite matrix $P$ such that
    \begin{align}
        \langle P \{D_vF(u;z)\} z ,z \rangle
        + \mu \langle PI_h z, z \rangle
        \geq 
        \frac{\alpha}{2}
        \|z\|^2.
    \end{align}

    The mean value theorem for vector valued functions states that
    \begin{equation}
        F(v;b) - F(u;b) = \int_0^1D_vF(u+\theta(v-u);b)(v-u)d\theta,
    \end{equation}
    so that
    \begin{equation}
        \langle F(v;b)-F(u;b),Pz\rangle = \left\langle\int_0^1D_vF(u+\theta z;b)z d\theta,Pz\right\rangle.
    \end{equation}
    Taking the inner product with $Pz$ of \eqref{eq:error_evolve1} we find
    \begin{align}
        \frac{1}{2}\frac{d}{dt}\|z\|^2_P &=  - \mu 
        \langle P I_h z,z \rangle - \langle F(v;b)-F(u;b), Pz\rangle
        \notag\\
        &- \langle F(u;b)-F(u;a), Pz\rangle 
        \notag\\ 
        &\leq 
        -\left\langle P \int_0^1D_vF(u+\theta z; b)\} zd\theta, z \right\rangle
        - \mu \langle PI_h z, z \rangle 
        \notag\\
        &+ \lambda_{\max}(P) C_F|b-a|\|z\|_P
        \notag\\
        &\leq
        - \frac{\alpha}{2} \| z\|^2
        + C_F \sqrt{\lambda_{\max}(P)} |b-a|\|z\|_P
        \\
        &\leq
        - \frac{\alpha}
        {2 \lambda_{\max}(P)}
        \|z\|^2_P
        + C_F
        \sqrt{\lambda_{\max}(P)}
        |b-a|\|z\|_P.
    \end{align}
    Applying Young's inequality we arrive at:
    \begin{align}
        \frac{1}{2}\frac{d}{dt}\|z\|^2_P 
        &\leq -\frac{\alpha}{4\lambda_{\max}(P)}
        \|z\|^2_P 
        +
        \frac{C_F^2 \lambda_{\max}^2(P)}
        {\alpha}
        |b-a|^2.
    \end{align}
    Letting $\gamma = \frac{\alpha}{2 \lambda_{\max}(P)}$ and $C_{b} = C_F^2\lambda_{\max}(P)$ we have 
    \begin{align}
        \frac{1}{2}\frac{d}{dt}\|z\|^2_P 
        &\leq -
        \frac{\gamma}{2}
        \|z\|^2_P 
        +
        \frac{C_b}{2\gamma}
        |b-a|^2.
    \end{align}
    Then Gronwall's inequality gives us 
    \begin{align}
        \|z(t)\|_P^2
        \leq \|z(0)\|_P^2 e^{-\gamma t}
        + \frac{C_b}{\gamma^2}
        (1 - e^{- \gamma t})|b-a|^2.
    \end{align}
    \enlargethispage{\baselineskip}% keep the end-of-proof box on this page
    Equivalence of the Euclidean and $P$-norms gives us
    \begin{align}
        \|z(t)\|^2
        &\leq 
        \frac{1}{\lambda_{\min}(P)}
        \|z(t)\|^2_P
        \notag\\
        &\leq 
        \frac{1}{\lambda_{\min}(P)}
        \|z(0)\|_P^2 e^{-\gamma t}
        + 
        \frac{1}{\lambda_{\min}(P)}
        \frac{C_b}{\gamma^2}
        (1 - e^{- \gamma t})|b-a|^2
        \notag\\
        &\leq
        \frac{1}{\lambda_{\min}(P)}
        \|z(0)\|^2 e^{-\gamma t}
        +
        \frac{1}{\lambda_{\min}(P)}
        \frac{C_b}{\gamma^2}
        |b-a|^2
    \end{align}
\end{proof}

This bound shows that as $t\rightarrow \infty$ then $v(t)\rightarrow u(t) + O(|b-a|)$ i.e. the solution will converge up to model error in the specification of the nudged system \eqref{eq:nudge}.  This is consistent with the results established in \cite{CHL_2020,farhat2020data,larios2020approximate} and further clarified in \cite{li2026continuoussurrogate} that continuous data assimilation works up to the error in the model assumptions.

We will also need a similar decay estimate for the linearized system (governing the sensitivity) given by
\begin{align}
    \frac{d x}{dt}
    + D_vF(v;b) x + \mu I_h x &= 0
    \label{eq:linearized}
    .
\end{align}
We will phrase this result in terms of the state transition matrix $\Phi(s,t)$ which satisfies
\begin{align}
        \frac{d}{dt}
        \Phi(t,s)
        + D_vF(v(t);b) \Phi(t,s)
        + \mu I_h \Phi(t,s)
        &= 0,
        \notag\\
        \frac{d}{ds}
        \Phi(t,s)
        -\Phi(t,s) D_vF(v(s);b)
        - \mu \Phi(t,s) I_h
        &= 0,
        \notag\\
        \Phi(s,s) = I.
        \label{eq:transitionMatrix}
\end{align}

\begin{theorem}
    For $\Phi(t,s)$ satisfying \eqref{eq:transitionMatrix}, with $\mu \geq \mu_0$ we have 
    \begin{align}
        \|\Phi(t,s)\| \leq C_\phi e^{-\gamma_\phi (t-s)},
    \end{align}
    for some positive, $\mu$-independent constants $C_\phi$ and $\gamma_\phi$  where the norm above is the operator norm on the matrix.
\end{theorem}
\begin{proof}
    Because $\Phi(t,s)$ is the time evolution operator we have 
    \begin{align}
        x(t) = \Phi(t,s) x(s),
    \end{align}
    for $t > s$ and $x(t)$ as the solution of \eqref{eq:linearized}.
    We just need to show that 
    \begin{align}
        \|x(t)\| \leq C_\phi  e^{- \gamma_\phi(t-s)} \|x(s)\|. 
    \end{align}
    For solutions of \eqref{eq:linearized} with any possible initial condition $x(s)$ starting at $t=s$, and constants $C_\phi$ and $\gamma_\phi$ chosen so that the inequality holds globally for all $t$ and $s$.
    Using assumption \ref{assumption:jacobian_dissipation}, we have
    \begin{align}
        \frac{d}{dt} \| x\|_P^2
        &= - 2\langle D_vF(v;b) x, Px \rangle
        - 2 \mu \langle I_h P x,x \rangle
        \notag\\
        &\leq - \alpha\| x\|^2
        \leq - \frac{\alpha}{ \lambda_{\max}(P)} \|x\|_P^2.
    \end{align}
    Gronwall's inequality gives us 
    \begin{align}
        \| x(t) \|_P  \leq
        e^{-\frac{\alpha}{2\lambda_{\max}(P)}(t-s)}
        \|x(s)\|_P.
    \end{align}
    Equivalence of the Euclidean and $P$-norms then gives us 
    \begin{align}
        \| x(t) \|  \leq
        \sqrt{\frac{\lambda_{\max}(P)}
        {\lambda_{\min}(P)}}
        e^{-\frac{\alpha}{2 \lambda_{\max}(P)}(t-s)}
        \|x(s)\|.
    \end{align}
    This holds for all initial conditions and thus we conclude that 
    \begin{align}
        \left\|\Phi(t,s)\right\| \leq
        \sqrt{\frac{\lambda_{\max}(P)}
        {\lambda_{\min}(P)}}
        e^{-\frac{\alpha}{2 \lambda_{\max}(P)}(t-s)}.
    \end{align}
\end{proof}

\section{Justification of the Asymptotic Approximation of Sensitivities}\label{sec:rigorous_asymptotics}
We consider the sensitivity equations for the nudged system.
Restated here, we define the sensitivity as the solution $W(t;b)$ of:
\begin{align}
    \dot{W}
    + D_vF(v;b) W + D_bF(v;b)
    + \mu I_h W = 0.
\end{align}
We will assume that the initial conditions are asymptotically small as $\mu \to \infty$.
Explicitly, we assume that the initial conditions satisfy
\begin{align}
    \| W(t=0)\| \leq \frac{1}{\mu} C_0
    \label{eq:InitCondit}
    .
\end{align}
If the initial conditions for the nudged system are chosen independent of previous choices of parameters, as proposed in \cite{ahmad2026dataassimilationaugmented}, this will always be true because we will have $W(t=0)=0$. Such a re-initialization can also be carried out with the open-source implementation accompanying \cite{schill2026asymptotic}, although no parameter-estimation results using it are reported there.
If the parameter updates are done in-line in epochs as in \cite{newey2025model}, then this will be true by induction; 
The first epoch will have $W(t=0) = 0$, and the inductive step follows from the results of this section (see section \ref{sec:inline}).

As in \cite{Lorenz_2021} we first need to show that the temporal derivative of the nudged solution is bounded.

\begin{lemma}[Bound on the time derivative with an initial layer] \label{lem:dvdt_bounded}
Let $z=v-u$ and suppose that the assumptions of the previous section
hold. Define
\[
E_0:=\|I_hz(0)\|,
\qquad
C_G:=L_FC_z+C_F|b-a|,
\]
where $\|z(t)\|\leq C_z$ uniformly in time. Then, for every $t\geq0$,
\[
\|I_hz(t)\|
\leq
e^{-\mu t}E_0
+
\frac{C_G}{\mu}\left(1-e^{-\mu t}\right).
\]
Consequently,
\[
\left\|\frac{dv}{dt}(t)\right\|
\leq
C_v+\mu E_0e^{-\mu t},
\qquad
C_v:=M_F+C_G,
\]
where
\[
M_F:=\sup_{t\geq0}\|F(v(t);b)\|.
\]
In addition, for every $\gamma>0$,
\[
\int_0^t
e^{-\gamma(t-s)}
\left\|\frac{dv}{ds}(s)\right\|\,ds
\leq
\frac{C_v}{\gamma}+E_0.
\]
All constants on the right-hand side are independent of $\mu$.
\end{lemma}
\begin{proof}
    The observable error satisfies
    \begin{equation}
        \frac{d q}{d t}  + \mu q = -I_h(F(v;b)-F(u;a)).
    \end{equation}
    Using variation of constants this is rewritten as
    \begin{equation}
        q(t) = e^{-\mu t}q(0) - \int_0^te^{-\mu(t-s)}I_h(F(v(s);b)-F(u(s);a))ds.
    \end{equation}
    Since
    \begin{equation}
        \|F(v;b)-F(u;a)\| \leq L_F\|v-u\| + C_F|b-a| \leq C_G,
    \end{equation}
    then we find that
    \begin{align}
        \|q(t)\| &\leq e^{-\mu t}E_0 + C_G\int_0^te^{-\mu (t-s)}ds\\
        &= e^{-\mu t}E_0 + \frac{C_G}{\mu}(1-e^{-\mu t}).
    \end{align}

    Now, we rewrite the evolution of $v(t)$ as
    \begin{equation}
        \frac{dv}{d t} = -F(v;b) - \mu q.
    \end{equation}
    Looking at the $L^2$ norm of this
    \begin{align}
        \left\|\frac{dv}{dt}\right\| &\leq M_F + \mu \|q\|,\\
        &\leq M_F + \mu E_0e^{-\mu t} + C_G(1-e^{-\mu t}),\\
        &\leq C_v + \mu E_0e^{-\mu t}.
    \end{align}

    Finally using these bounds, we get to
    \begin{align}
        \int_0^te^{-\gamma (t-s)}\left\|\frac{dv}{ds}(s)\right\| ds &\leq C_v\int_0^te^{-\gamma (t-s)}ds + \mu E_0\int_0^te^{-\gamma (t-s)}e^{-\mu s}ds,\\
        &\leq \frac{C_v}{\gamma} + E_0,
    \end{align}
    where we used the fact that $e^{-\gamma (t-s)} \leq 1$ for the second integral above.
\end{proof}

\begin{lemma} \label{lemma:Wsmall}
    Applying the same assumptions as above, we have 
    \begin{align}
        \|W(t)\| \leq \frac{1}{\mu} C_w,
    \end{align}
    for some constant $C_w$ independent of $\mu$.
\end{lemma}
\begin{proof}
    Consider the sensitivity equation 
    \begin{align}
        \frac{dW}{dt}
        + D_vF(v;b) W + \mu I_h W
        &= - D_bF(v;b),
        \\
        W(0) = W_0.
    \end{align}
    where we remind the reader of the assumption that 
    \begin{align}
        \|W_0\| \leq \frac{1}{\mu} C_0
    \end{align}
    
    Let $\Phi(t,s)$ be the state transition matrix for this linear system as defined by \eqref{eq:transitionMatrix}. Duhamel's principle tells us the general solution will then be 
    \begin{align}
        W = \Phi(t,0) W_0
        - \int_0^t \Phi(t,s) D_bF(v(s);b)
        ds.
    \end{align}
   On the initial condition dependent term;
    we have
    \begin{align}
        \|\Phi(t,0) W_0\|
        \leq 
        \frac{1}{\mu}
        C_\phi e^{-\gamma_\phi t} C_0.
    \end{align}
    Next we consider the inhomogeneous term.
    We have $D_bF = I_h D_bF$ which gives us
    \begin{align}
        &\int_0^t
        \Phi(t,s)D_bF(v(s);b) ds
        =
        \int_0^t
        \Phi(t,s) 
        I_h
        D_bF(v(s);b) ds
        \notag\\
        &= 
        \frac{1}{\mu}
        \int_0^t
        \left(
        \frac{\partial}{\partial s}
        \Phi(t,s)
        - \Phi(t,s) D_vF(v(s);b)
        \right)
        D_bF(v(s);b)
        ds.
    \end{align}
    Integrating by parts this integral can be written as 
    \begin{align}
        &= \frac{1}{\mu}
        \left[D_bF(v(t);b) 
        - \Phi(t,0) D_bF(v(0);b)
        \right]     
        - \frac{1}{\mu}
        \int_0^t \Phi(t,s)
        \frac{\partial}{\partial s}
        D_bF(v(s);b)
        ds
        \notag\\
        &+ 
        \frac{1}{\mu}
        \int_0^t
        \Phi(t,s) D_vF(v(s);b)
        D_bF(v(s);b)
        ds.
    \end{align}
    Looking at each of these terms separately we have
    \begin{align}
        \left\|\frac{1}{\mu}
        \left(
        D_bF(v(t);b)
        - \Phi(t,0) 
         D_bF(v(0);b)
        \right)
        \right\|
        \leq \frac{1}{\mu}
        \left(
        1 + C_\phi e^{-\gamma_\phi t}
        \right) C_F.
    \end{align}
    Moving on we have
    \begin{align}
        \left\| 
        \frac{1}{\mu}
        \int_0^t \Phi(t,s)
        \frac{\partial}{\partial s}
        D_bF(v(s);b)
        ds
        \right\|
        &\leq
        \frac{1}{\mu}
        \int_0^t C_\phi e^{-\gamma_\phi (t-s)}
        \left\|\frac{\partial}{\partial s}
        D_bF(v(s);b)
        \right\|ds.
    \end{align}
    Notice that 
    \begin{align}
        \frac{\partial}{\partial s}
        D_bF(v(s);b)
        = D_b\{D_vF(v(s);b)\}
        \frac{d v}{ds},
    \end{align}
    so that we can bound this factor as
    \begin{align}
        \left\|
        D_b\{D_vF(v(s);b)\}
        \frac{d v}{ds}
        \right\|
        \leq
        L_{Fb}
        \left\|\frac{dv}{ds}\right\|.
    \end{align}
    This gives us 
    \begin{align}
        \left\| 
        \frac{1}{\mu}
        \int_0^t \Phi(t,s)
        \frac{\partial}{\partial s}
        D_bF(v(s);b)
        ds
        \right\|
        &\leq 
        \frac{C_\phi L_{Fb}}{\mu}\int_0^te^{-\gamma_\phi(t-s)}\left\|\frac{dv}{ds}\right\|ds\\
        &\leq \frac{C_\phi L_{Fb}}{\mu}\left(\frac{C_v}{\gamma_\phi} +E_0\right).
    \end{align}
    The final remaining term is bounded as 
    \begin{align}
        \left\|\frac{1}{\mu}
        \int_0^t
        \Phi(t,s) D_vF(v(s);b)
        D_bF(v(s);b)
        ds
        \right\|
        &\leq \frac{1}{\mu}
        \frac{1 - e^{-\gamma_\phi t}}{\gamma_\phi}
        C_\phi
        L_F C_F,
    \end{align}
    which completes the proof.
\end{proof}

We can now prove the desired result
\begin{theorem}\label{thm:asymptotic_sensitivity}
    For $t > \frac{2}{\mu} \ln \mu,~\mu>2$, with $\mu\geq \mu_0$, and given the previous assumptions we have 
    \begin{align}
        \left\| I_hW(t) + \frac{1}{\mu}
        D_bF(v(t);b)
        \right\|
        \leq 
        \frac{1}{\mu^2}
        C_\zeta.
    \end{align}
\end{theorem}
\begin{proof}
    Define the asymptotic error in the representation as
    \begin{equation}
        \zeta = I_hW + \frac{1}{\mu} D_bF.
    \end{equation}
    The time derivative of $\zeta$ is then given by
    \begin{align}
        \frac{d\zeta}{dt} &= I_h\frac{dW}{dt} + \frac{1}{\mu}\frac{d}{dt}
        D_bF
        \\
        &= - \mu \zeta - I_h D_vF(v;b)W + \frac{1}{\mu}\frac{d}{dt}
        D_bF
        ,
    \end{align}
    where we have used the fact that $I_h\zeta = \zeta$.
Taking the inner product with $\zeta$ leads to
\begin{equation}
    \frac{1}{2}\frac{d}{dt}\|\zeta\|^2 = -\mu \|\zeta\|^2 - \left\langle I_h  D_vF(v;b)w,\zeta\right\rangle + \frac{1}{\mu}\left\langle \frac{d}
    {dt}
    D_bF
    ,\zeta\right\rangle.
\end{equation}

Before proceeding further, we note that
\begin{equation}
    \frac{d}{dt}
    D_bF
    = D_b\{D_vF(v;b)\}\frac{dv}{dt}.
\end{equation}
Then it follows from Lemma \ref{lem:dvdt_bounded} that 
\begin{equation}
    \left\|\frac{d}{dt}
    D_bF
    \right\| \leq L_{Fb}
    \left( C_v +   \mu e^{-\mu t}E_0
    \right).
\end{equation}

From Lemma \ref{lemma:Wsmall} 
we also see that $\|W\| \leq \frac{C_w}{\mu}$ so that
\begin{align}
\left\|I_hD_vF(v;b)W\right\| &\leq \|D_vF(v;b)W\|
\leq L_F\frac{C}{\mu}.
\end{align}
Combining all of this together yields
\begin{align}
    \frac{d}{dt}\|\zeta\| &\leq -\mu\|\zeta\| + 
    \frac{1}{\mu} \left( L_F C_w +
    L_{Fb} C_v
    \right)
    +L_{Fb} E_0 e^{-\mu t} 
    .
\end{align}
Multiplying both sides of this inequality by $e^{\mu t}$ and integrating, we arrive at
\begin{equation}
    \|\zeta(t)\| \leq \|\zeta(0)\|e^{-\mu t} + 
    \frac{1}{\mu^2} \left( 
     L_F C_w +
    L_{Fb} C_v
    \right)
    + L_{Fb} E_0 t e^{-\mu t}
    .
    \label{eq:AsymFinalResult}
\end{equation}
Then we have 
\begin{align}
    \| \zeta(0)\|
    &= \left\| I_hW(0)
    + \frac{1}{\mu}D_bF
    \right\|
    \notag\\
    &\leq 
    \frac{1}{\mu} 
    \left(
    C_0
    + C_F
    \right).
\end{align}
For $t\geq \frac{2}{\mu} \ln \mu$, and using $\mu>2$, we have
\begin{align}
    e^{-\mu t} \leq \frac{1}{\mu^2},
    \quad te^{-\mu t} \leq \frac{2\ln \mu}{\mu^3}
    \leq \frac{1}{\mu^2}
\end{align}
which gives
\begin{align}
    \|\zeta(t)\| \leq 
    \frac{1}{\mu^2}
    \left(
    \frac{1}{\mu}(C_0 + C_F)
    + L_FC_w + L_{Fb}\left( 
    C_v + E_0
    \right)
    \right).
\end{align}
\end{proof}

\section{Optimization Guarantees with Asymptotic Gradients}\label{sec:rigorous_optimize}

Define the asymptotic sensitivity as
\begin{align}
    I_hW_{\text{asymp}}
    = 
    -\frac{1}{\mu} I_h\{D_bF\}.
\end{align}
As we have established, 
\begin{align}
    \| I_hW - I_hW_{\text{asymp}} \|
    \leq \frac{C_\zeta}{\mu^2},
\end{align}
where $W$ is the Jacobian sensitivity and this is established
for large $\mu$, where $C_\zeta$ is a constant independent of $\mu$.
This leads to the following approximation of the gradient of the error function 
\begin{align}
    \nabla_{\text{asymp}}
    E = \langle I_hW_{asymp}, I_h( v- u) \rangle.
\end{align}
This will have the following error estimate 
\begin{align}
    |\nabla E - \nabla_{\text{asymp}}E |
    \lesssim \frac{1}{\mu^2}
    \|I_h(v-u)\|
    = \frac{1}{\mu^2} \sqrt{ 2 E}
    ,
\end{align}
which tells us that at the minimum we have 
\begin{align}
    \|\nabla E - \nabla_{\text{asymp}}E \|_{b = \text{argmin} E}
    \lesssim \frac{1}{\mu^2}
     \sqrt{ \min E}.
\end{align}

From this, we see that using the modified gradient rather than the true gradient modifies the convergence of standard gradient based optimization methods.
Because of this we can no longer expect a theoretical guarantee of exact convergence to the true minimizer.
Instead if we write $b^* = \text{argmin}E$ we find convergence 
\begin{align}
    E(b^k)
    - E(b^*)
    \lesssim \frac{1}{\mu^2} E(b^*),
\end{align}
for $k$ chosen large enough as shown below. 
Despite this, in practice we find that for appropriately chosen $\mu$ and $T$ 
we often find convergence to within machine error of the correct value of the parameter (see e.g. \cite{newey2025model}).
We will restrict our attention to gradient descent, as other first and second order methods yield similar results via a messier calculation.

\subsection{Polyak-Lojasiewicz inequality}
If the loss function $E(b)$ is strictly (or even strongly) convex then convergence of gradient descent under a suitable choice of the learning rate is nearly automatic.  Rather than relying on the stricter condition of convexity, we will focus on loss functions which satisfy the Polyak-Lojasiewicz (PL) inequality with constant $0 < \lambda < \infty$, i.e.
\begin{equation}\label{eq:PL_condition}
    E(b) - E(b^*) \leq \frac{1}{2\lambda}|\nabla E(b)|^2.
\end{equation}

Loss functions which satisfy the PL-condition are nearly convex in a local sense (it is easily shown that strong convexity implies \eqref{eq:PL_condition}).  
Hence, we will focus on the PL-condition rather than any of several convexity conditions (see \cite{garrigos2023handbook} for a thorough description).  
Because the data assimilation algorithm can only be run over a finite time window $E$ will not have an exact zero at the correct value of the parameter. 
However, it generally will still be true that $E(b^*)$ is very small (e.g. at the order of machine error) so long as $T$ is chosen large enough.
In this case we can also consider, rather than an exact PL condition, an approximate PL condition given by
\begin{align}
     \left(
    E(b) -E(b^*)
    \right)
    \leq \frac{1}{2\lambda } |\nabla E(b) |^2
    + C E(b^*).
\end{align}
In the cases we consider, the constant $C = \mathcal{O} (1/\mu^2)$ which gives us further control over the remainder term.

\begin{theorem} \label{thm:PL_approx}
    For the loss function \eqref{eq:error_func} coupled to the system \eqref{eq:truth}-\eqref{eq:nudge} which satisfies assumptions     \ref{assumption:projection_obs}-\ref{assumption:jacobian_dissipation} with $\mu >\max\{\mu_0,2\}$ and $t > \frac{2}{\mu}\ln\mu$, and where the asymptotic representation of the observed sensitivity $I_h D_b v(t;b) = I_h W \sim  \frac{1}{\mu}G(t;b)$, where $G(t;b) = - I_h D_b F(v(t);b)$ also satisfies the observable identifiability condition 
    \begin{align}
        |\langle G,I_h(v-u) \rangle|^2
        \geq \chi 
        \left( \|I_h(v-u)\|^2\big|_b
        - \|I_h(v-u)\|^2\big|_{b^*}
        \right)
        = 2\chi 
        \left(
        E(b) - E(b^*)
        \right)
        ,
        \label{eq:asymp_grad_obs}
    \end{align}
    for some $0 < \chi < \infty$ and all $b \in B$, then 
    \begin{align}
        \left(
        \frac{\chi}{2 \mu^2}
        - \frac{C_\zeta^2}{\mu^4}
        \right)
        \left( E(b) -E(b^*) \right)
        \leq 
        \frac{1}{2} |\nabla E(b)|^2
        + \frac{C_\zeta^2}{\mu^4} E(b^*).
        \label{eq:approxPL_mu}
    \end{align}
\end{theorem}
\begin{proof}
    First write out the sensitivity as $I_hw = \frac{1}{\mu}G + \zeta$ where we know from Theorem \ref{thm:asymptotic_sensitivity} that $\|I_h\zeta\| \leq \frac{C_\zeta}{\mu^2}$ for some finite positive constant $C_\zeta$. 
    It follows that we can write the gradient of the loss function as 
    \begin{align}
        \nabla E(b) = \frac{1}{\mu} \langle G , I_h (v-u) \rangle
        + \langle\zeta, I_h(v-u) \rangle.
    \end{align}
    Using Young's inequality, and Cauchy-Schwarz for generic vectors $x$ and $y$ we have
    \begin{align}
        \| x+ y\|^2 &\geq \|x\|^2 + 2 \langle x,y \rangle + \|y\|^2
        \notag\\
        &\geq \|x\|^2 + \|y\|^2 -2 \|x\| \|y\|
        \notag\\
        &\geq \frac{1}{2}\|x\|^2 -\|y\|^2.
    \end{align}
    Applying this to our problem we have
    \begin{align}
        \frac{1}{2}|\nabla E(b) |^2 
        &\geq \frac{1}{4 \mu^2}
        | \langle G,I_h(v-u) \rangle|^2
        - 
        \frac{1}{2}
        |\langle \zeta, I_h(v-u) \rangle |^2
        \notag\\
        &\geq 
        \frac{\chi}{4 \mu^2}
        \left( \|I_h(v-u)\|^2\big|_b
        - \|I_h(v-u)\|^2\big|_{b^*}
        \right)
        - \frac{C_\zeta^2}{2\mu^4}
        \|I_h(v-u)\|^2\big|_b
        \notag\\
        &=
        \frac{\chi}{2\mu^2}
        \left(
        E(b) - E(b^*)
        \right)
        - \frac{ C_\zeta^2}{\mu^4}
        E(b)
        \notag\\
        &=
        \left(
        \frac{\chi}{2\mu^2}
        -\frac{C_\zeta^2}
        {\mu^4}
        \right)
        \left(
        E(b) - E(b^*)
        \right)
        - \frac{ C_\zeta^2}{\mu^4} E(b^*)
        .
    \end{align}
\end{proof}

This result indicates that so long as the underlying dynamical system is sufficiently smooth and satisfies the primary Assumptions, then the prescribed $L^2$ loss function satisfies an (approximate) PL-condition.  
This holds as long as $\mu$ is sufficiently large so that the asymptotic representation of the sensitivity given in Theorem \ref{thm:asymptotic_sensitivity} holds, but as $\mu\rightarrow \infty$ then the constant in the PL-condition converges to $0$.  
This indicates that while the resultant gradient approximation is improved for large $\mu$, the resulting loss is flatter, i.e. gradient descent or a similar algorithm will need to select a larger step size for the optimization procedure to converge in a reasonable number of steps, particularly near the optimal value $b^*$.

\subsection{Gradient descent with the PL-condition}
We will consider gradient descent with the asymptotic approximation to the gradient, i.e.
\begin{align}
    b_{i,k+1} &= b_{i,k} - \gamma_k \nabla_{b_i,asymp}E(b_k)\\ \label{eq:asymp_grad_descent}
    &= b_{i,k} - \gamma_k\nabla_{b_i} E(b_k) + \gamma_k \langle v - u ,I_h \zeta_{b_i} \rangle.
\end{align}
In addition to the PL-condition, convergence of gradient descent will also require L-smoothness (Lipschitz of the gradient) of the loss $E(b)$.  For our setting, this is implied if we require that the first and second order sensitivities are bounded for all parameter choices $b_k$.

\begin{theorem} \label{thm:gradDescentConv}
    Suppose that $E(b)$ is $L_d$-smooth and satisfies the approximate PL-condition 
    of Theorem \ref{thm:PL_approx}
    with constant $0 < \lambda = \frac{\chi}{2\mu^2} - \frac{C_\zeta^2}{\mu^4} < L_d$.  
    If the learning rate at each step satsifies $\gamma_k  < \frac{1}{L_d} < \frac{1}{\lambda}$ 
    and $\mu$ is chosen so that
    \begin{equation}
        \mu^2 > \frac{4 C_\zeta^2}{\chi},
    \end{equation}
    then
    \begin{equation}
        E(b_{k+1}) - E(b^*) \leq \left(1- \gamma_k\frac{\chi \mu^2 - 4C^2_\zeta}{2\mu^4}\right) (E(b_k) - E(b^*)) 
        + \gamma_k \frac{2C^2_\zeta}{\mu^4}E(b^*),
    \end{equation}
    which implies that so long as the learning rate is also chosen at each step to satisfy \begin{equation}\label{eq:learning_rate_restrict}
        \gamma_k < \frac{2\mu^4}{\chi\mu^2 - 4C^2_\zeta},
    \end{equation}
    and $\gamma_k > \gamma >0$  (i.e. $\gamma_k$ is bounded uniformly away from zero) for all $k$ then
    \begin{equation}
        \limsup_{k\rightarrow\infty}\left[E(b_k)-E(b^*)\right] \leq \frac{4C^2_\zeta}{\chi\mu^2 - 4C^2_\zeta}E(b^*).
    \end{equation}
\end{theorem}
\begin{proof}
    To simplify the presentation of this proof, we will use $\cdot$ to indicate the dot product on parameter space. 
    First, $L_d$-smoothness provides us with a descent lemma (see \cite{garrigos2023handbook}) which states that
    \begin{align}
        E(b_{k+1}) &\leq E(b_k) +  \nabla E(b_k) \cdot (b_{k+1}-b_k) + \frac{L_d}{2}
        |b_{k+1} - b_k|^2.
    \end{align}
    Inserting the asymptotic gradient descent update for $b_{k+1}$ given by \eqref{eq:asymp_grad_descent} and rearranging terms appropriately, we see that
    \begin{align}
        E(b_{k+1}) &\leq E(b_k) - \gamma_k |\nabla E(b_k)|^2 + \frac{L_d}{2}\gamma^2_k
        |\nabla E(b_k)|^2\\
        &+ \gamma_k (1-L_d\gamma_k)
        \nabla E(b_k) \cdot 
        \langle
        v - u, I_h \zeta \rangle 
        + \frac{L_d\gamma^2_k}{2}
        | \langle v- u, I_h \zeta \rangle |^2
        \notag\\
        &\leq 
       E(b_k) - \gamma_k |\nabla E(b_k)|^2 
       + \frac{L_d}{2}\gamma^2_k |\nabla E(b_k)|^2
        \notag\\
        &+
        \frac{C_\zeta}{\mu^2}
        \gamma_k 
        (1-L_d \gamma_k) |\nabla E(b_k)|
        \sqrt{2E(b_k)}
        + L_d\gamma^2_k
        \frac{C^2_\zeta}{\mu^4}
        E(b_k)
        .
    \end{align}
    Combining like terms, and using Young's inequality with weight $\frac{C}{\mu^2}$ leads to
    \begin{align}
        E(b_{k+1}) &\leq 
        \left(1 + \frac{\gamma_k C^2_\zeta}{\mu^4}\right)
        E(b_k) 
        - \frac{\gamma_k}{2}|\nabla E(b_k)|^2.
    \end{align}
    Now, using the approximate PL-condition given in equation \eqref{eq:approxPL_mu} we have
    \begin{align}
        &E(b_{k+1}) - E(b^*) 
        \leq
        \left(1 + \frac{\gamma_k C^2_\zeta}{\mu^4}\right)
        E(b_k)  
        - E(b^*)
        \notag\\
        &-\gamma_k
        \left(\frac{\chi}{2\mu^2}
        - \frac{C_\zeta^2}{\mu^4}
        \right)
        \left( E(b_k) - E(b^*)\right)
        + \gamma_k\frac{C_\zeta^2}{\mu^4}E(b^*)
        \notag\\
        &\leq \left(1 -\gamma_k \frac{\chi \mu^2-  4 C^2_\zeta}{2\mu^4}\right)
        \left( E(b_k) - E(b^*) \right)
        + \gamma_k\frac{2 C_\zeta^2}{\mu^4}E(b^*).
    \end{align}
\end{proof}

\subsection{In-line Parameter Estimation} \label{sec:inline}

The approach outlined so far in this paper is to apply parameter estimation by running the nudging algorithm repeatedly over the same time interval.
This is primarily to simplify the presentation.
This general approach for estimating parameters was originally developed in-line with the continuous data assimilation algorithm \cite{newey2025model}.
The approach described in \cite{newey2025model} (following approaches in \cite{CHL_2020,Lorenz_2021}) was to apply the parameter estimation in epochs (see \cite{pachev2022concurrentmultiparameterlearningdemonstrated} for a related approach in which the parameter estimates are instead updated continuously in time).
The data assimilation algorithm was run over intervals, where the system was allowed to relax according to Lemma \ref{lem:bounded_error}.
At the end of the interval a parameter estimation step was applied using the asymptotic gradient approximation described here.

If we assume that the assumptions used in this paper hold uniformly for all time rather than over a fixed time interval, then it is straightforward to extend the asymptotic gradient used here.
We sketch this approach below; consider a sequence of error functions $\{E_k(b)\}$ defined by 
\begin{align}
    E_k(b) = \frac{1}{2}
    \|I_h(v(t;b) -u(t))\|^2 
    \bigg|_{t=t_k}, 
\end{align}
where $\{t_k\}$ is an increasing sequence of times with $t_k - t_{k-1} > \frac{2}{\mu} \ln\mu$ (it takes the place of $T$ in this paper).

The initial condition for $W$ at the beginning of each epoch will no longer be zero. Instead we will have 
\begin{align}
    D_{b_{k+1}} v(t_k;b_{k})
    = 
    D_{b_{k}}v(t_k;b_k) D_{b_{k+1}} b_k .
\end{align}
If we use the gradient descent algorithm (other optimization algorithms are similar) we have 
\begin{align}
    b_{k+1} = b_k + \frac{\gamma_k}{\mu}  
    \langle v(t_k,b_k) - u(t_k), D_b F \rangle.  
\end{align}
This gives 
\begin{multline}
    D_{b_{k}} b_{k+1}
    = I \\
    + \frac{\gamma_k}{\mu}
    \left(
    \langle D_{b_k} v(t_k,b_k),
    D_{b_k} F \rangle
    + \langle v(t_k,b_k) - u(t_k),
    D_{b_k}^2 F 
    + [D_v D_{b_k} F] W
    \rangle
    \right).
\end{multline}
So long as the quantity in the parentheses is bounded (requiring additional regularity on $F$ and its derivatives) then
\begin{align}
    D_{b_k} b_{k+1} -I = \mathcal{O}\left(\frac{1}{\mu}\right).
\end{align}
Using a Neumann series we can see that 
\begin{align}
    D_{b_{k+1}} b_k - I= 
    (D_{b_k} b_{k+1})^{-1} -I = \mathcal{O}\left(\frac{1}{\mu}\right),
\end{align}
so long as $\gamma_k$ is chosen order 1 in $\mu$.
By Lemma \ref{lemma:Wsmall} the sensitivity at the end of each epoch will be $\mathcal{O}(1/\mu)$, and 
from this, the initial conditions of the sensitivity at the beginning of each epoch will also be $\mathcal{O}(1/\mu)$, and the results of section \ref{sec:rigorous_asymptotics} apply.
Then by our assumptions, Theorem \ref{thm:PL_approx} holds with constants defined uniformly, we can apply Theorem \ref{thm:gradDescentConv} to each epoch's error function $E_k$ and each step of the parameter estimation.
A challenge still remains because each $E_k$ will be a different function potentially with a different minimizer $b^*_k$.
We leave the work of extending the proof of gradient descent to the in-line approach to future work.

\section{Assumptions applied to Lorenz '63}\label{sec:examples}
To clarify the impact of the assumptions used in this paper, we consider the specific example of the Lorenz '63 system \cite{lorenz1963deterministic} where \eqref{eq:truth} is given by
\begin{align}
    \frac{dx}{dt} &= \sigma (y-x),\\
    \frac{dy}{dt} &= x(\rho-z)-y,\\
    \frac{dz}{dt} &= xy-\beta z,
\end{align}
where the observation operator is defined as
\begin{equation}
    I_h\begin{pmatrix}
        x\\y\\z
    \end{pmatrix} =
    \begin{pmatrix}
        0 & 0 & 0\\
        0 & 1 & 0\\
        0 & 0 & 0
    \end{pmatrix}\begin{pmatrix}
        x\\y\\z
    \end{pmatrix} = \begin{pmatrix}
        0\\y\\0
    \end{pmatrix},
\end{equation}
i.e. we can observe only the $y$ variable.  For convergence of nudging, observation of other combinations of $x,y,$ and $z$ are also valid except for observation of $z$ only.  The other cases follow similar to the derivation provided here so we omit the details.

We now turn to parameter estimation on this system.
We seek to estimate the parameter $\rho$, this is the only parameter we can estimate given the observation operator chosen here due to the restriction $J_h D_bF = 0$.
Making different choices allows for different sets of parameters to be estimated (see e.g. \cite{Lorenz_2021}).

To clarify the notation set up in \eqref{eq:truth}, we note that for this specific problem,
\begin{align}
    F(x,y,z;\rho) = -\begin{pmatrix}
        \sigma (y-x)\\
        x(\rho-z) - y\\
        xy-\beta z
    \end{pmatrix}.
\end{align}
In this setting, the sensitivities are defined as
\begin{equation}
    w_1 = \frac{\partial x}{\partial \rho},\quad w_2 = \frac{\partial y}{\partial \rho}, \quad w_3 = \frac{\partial z}{\partial \rho},
\end{equation}
and correspondingly satisfy
\begin{align}
    \frac{dw_1}{dt} &= \sigma (w_2 - w_1),\\
    \frac{dw_2}{dt} &= x(1-w_3) + w_1(\rho - z) - w_2 - \mu w_2,\\
    \frac{dw_3}{dt} &= xw_2 + w_1y - \beta w_3.
\end{align}
Notice that the Jacobian is (in the notation of the previous sections)
\begin{equation}
    -D_vF =
    \begin{pmatrix}
        \sigma & -\sigma & 0\\
        -\rho + z & 1 & x\\
        -y & -x & \beta
    \end{pmatrix},
\end{equation}
with inhomogeneous term $(0,x,0)^T$.

\subsection{Connection with key Assumptions}
Assumption \ref{assumption:projection_obs} is clearly satisfied in this setting where the observation operator and orthogonal complement can be defined as:
\begin{equation}
    I_h = \begin{pmatrix}0 & 0 & 0\\ 0 & 1 & 0\\ 0 & 0 & 0\end{pmatrix},\quad J_h = \begin{pmatrix}
        1 & 0 & 0\\0 & 0 & 0\\0 & 0 & 1
    \end{pmatrix}.
\end{equation}
Similarly Assumption \ref{assumption:observable_parameters} is easily verified:
\begin{equation}
    J_h D_\rho F = J_h\frac{\partial F}{\partial \rho} = \begin{pmatrix}
        1 & 0 & 0\\
        0 & 0 & 0\\
        0 & 0 & 1
    \end{pmatrix}\begin{pmatrix}
        0\\ -x\\ 0
    \end{pmatrix} = \begin{pmatrix}
        0\\0\\0
    \end{pmatrix}.
\end{equation}
For the Lorenz '63 system every solution of \eqref{eq:truth} eventually enters, and remains in, a bounded ellipsoid (see \cite[Appendix C]{sparrow2012lorenz}), and a similar energy estimate bounds the difference between the nudged and reference solutions (\cite[Section 4.1]{Lorenz_2021} proves such an estimate when $x$ is nudged, and \cite[Remark 4.7]{Lorenz_2021} states that slight modifications cover the other observed combinations); together these give Assumption \ref{assumption:absorbing_ball}, and Assumption \ref{assumption:technical} then follows from the polynomial form of $F$ on the resulting bounded set.  

To establish Assumption \ref{assumption:jacobian_dissipation} we demonstrate that this system satisfies the hypotheses of Lemma \ref{lem:jacobian_dissipation} as follows.  First, for any vector $\Pi = (\pi_1,\pi_2,\pi_3)^T$ note that
\begin{equation}
    \|J_h\Pi\|^2 = \pi_1^2+\pi_3^2.
\end{equation}
Now,
\begin{equation}
    P_J = \begin{bmatrix}
        \Lambda & 0 & 0\\0 & 0 & 0\\
        0 & 0 & 1
    \end{bmatrix},
\end{equation}
so that
\begin{align}
     J_h\{D_vF\}J_h &= \begin{bmatrix}\sigma & 0 & 0 \\
    0 & 0 & 0\\
    -y & 0 & \beta
    \end{bmatrix}\\
    \Rightarrow P_JJ_h\{D_vF\}J_h\Pi &= \begin{bmatrix}
        \sigma \Lambda & 0 & 0\\
        0 & 0 & 0\\
        -y & 0 & \beta
    \end{bmatrix}\Pi = \begin{bmatrix}
        \sigma\Lambda \pi_1\\
        0\\
        -y\pi_1 + \beta\pi_3
    \end{bmatrix},\\
    \Rightarrow \langle P_JJ_h\{D_vF\}J_h \Pi,J_h \Pi\rangle &= \sigma \Lambda \pi_1^2 - y\pi_1\pi_3 + \beta\pi_3^2 \geq \frac{\gamma}{2}\|J_h\Pi\|^2,
\end{align}
so long as 
\begin{equation}
    \Lambda > \frac{Y^2}{4\sigma\beta},\quad\mbox{and}\quad \gamma = \sigma\Lambda + \beta - \sqrt{(\sigma\Lambda-\beta)^2 +Y^2},
\end{equation}
where $Y$ is chosen such that
\begin{equation}
    |y(t)| < Y,
\end{equation}
for all $t>0$ and all reasonable choices of $\rho$.

Finally, we note that this setting also satisfies the localized observable identifiability condition as well.  In this case,
\begin{equation}
    G = -I_h D_\rho F = \begin{bmatrix}
        0 \\ x \\ 0
    \end{bmatrix},
\end{equation}
so that if $v = (x,y,z)^T$ and $u=(\tilde{x},\tilde{y},\tilde{z})^T$ then this condition becomes
\begin{align}
    G^TI_h(v-u) &= x(y-\tilde{y}),\\
    \Rightarrow \|G^TI_h(v-u)\|^2 &= x^2(y-\tilde{y})^2 ,
\end{align}
and since this is evaluated at the final time $t=T$ this is valid only at times when $x(T)$ is bounded away from zero.  In other words the identifiability of $\rho$ is valid for $y$-observations only for end times $t=T$ when $x(T)$ is not near the origin which is akin to the identifiability condition utilized in \cite{Lorenz_2021}.

\subsection{Comparison with key Theorems}
Figure \ref{fig:gradient_comparisons} compares the observable asymptotic sensitivity against a direct simulation of the observable sensitivity equations for the Lorenz '63 system, demonstrating the scaling predicted by Theorem \ref{thm:asymptotic_sensitivity}.  The figure displays three different options.  For parameter estimation, full observation of all the state variables is a trivial exercise, but we include this case in Figure \ref{fig:gradient_comparisons} to demonstrate the scaling of the asymptotic sensitivity anyways.  Note that the asymptotic representation of the sensitivity performs much better for complete observations and even when observing 2/3 of the state variables.  As shown in the right panel, the $O(\mu^{-2})$ scaling is apparent even when observing only one of the three state variables, but the scaling doesn't hold until higher $\mu$ values.

\begin{figure}
    \centering
    \includegraphics[width=\linewidth]{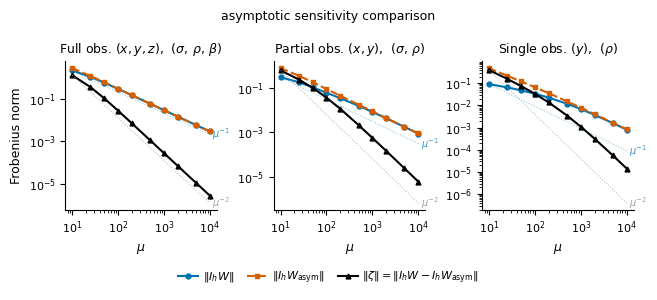}
    \caption{Comparison of the asymptotic sensitivities \eqref{eq:asymptotic} against the sensitivity computed via direct simulation of \eqref{eq:sensitivity} for the Lorenz '63 system.  Note that both the asymptotic sensitivity and the simulated sensitivity scale as $O(\mu^{-1})$ and the error between the approximations scale as $O(\mu^{-2})$.  In each case, the unknown parameters are exactly those that fall on the observable part of the state evolution equation.}
    \label{fig:gradient_comparisons}
\end{figure}

Figure \ref{fig:PL_comparison} demonstrates the realizability of the conditions in Theorem \ref{thm:PL_approx} by numerically computing the ratio of the squared gradient of the loss function relative to two times the loss function itself (in this case $E(b^*) \approx 0$).  As in the comparison described for Figure \ref{fig:gradient_comparisons}, the more part of the state that is observed, the closer the asymptotic result stated in Theorem \ref{thm:PL_approx} is to being realized.

\begin{figure}
    \centering
    \includegraphics[width=\linewidth]{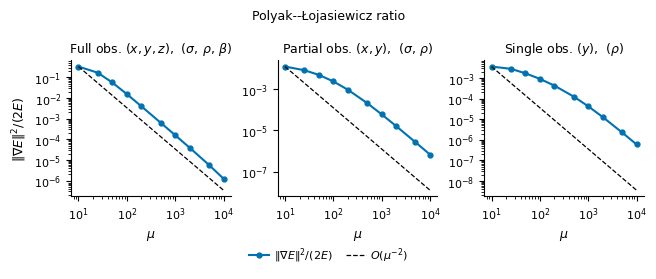}
    \caption{Numerical comparison for the PL-inequality for the Lorenz '63 system.}
    \label{fig:PL_comparison}
\end{figure}

Both of these Figures only consider two different circumstances where the observations aren't complete, but similar behavior is seen for the observable sensitivities (i.e. those with $J_h D_b F =0$) for other combinations of the observable state such as observations of $x$ and $z$ or observations of $x$ only.  As is well noted for this system (see \cite{balmforth1997master} for instance), observation of $z$ only is not sufficient for recovering the full state due to the inherent symmetry in the attractor for this system.

\section{Conclusions}\label{sec:conclusions}
We have presented a rigorous justification for a modification of the algorithm developed in \cite{newey2025model}.  The rigorous results presented here are stated for parameter updates over the same time interval, i.e. we have assumed that the state is simulated over the same interval $[0,T]$ between each update whereas \cite{newey2025model} focused on parameter updates on the fly or as the data was observed over a changing time interval.  Numerically, in experiments not reported here, we have seen very little, if any, difference between these two approaches. 
The PL-condition follows immediately from the identifiability condition once Theorem \ref{thm:asymptotic_sensitivity} is established.
We leave the rigorous justification of the sequential parameter updates for future work.  

Replacing \eqref{eq:error_func} with a time average as in \cite{schill2026asymptotic,ahmad2026dataassimilationaugmented} is another possibility not explicitly addressed here.  Even so, we note that all of the results in Section \ref{sec:rigorous_asymptotics} will still apply for a time-averaged loss as this doesn't change the need to approximate the parameter sensitivities, and similar estimates to those provided in Section \ref{sec:rigorous_optimize} will hold with some minor modifications for the loss function itself.

It is also worth noting that the asymptotic results of Section \ref{sec:rigorous_asymptotics} are not restricted to an $L^2$ loss function such as \eqref{eq:error_func} but provide a reasonable surrogate gradient for any other loss function dependent on the observable error.  This opens up the possibility of using the approximate sensitivities for other losses where physical considerations warrant a different type of loss.

\section*{Acknowledgments}
JPW was partially supported by NSF grant DMS-2510495 and CCF-2343286.  ChatGPT and Claude were used as proof checks, and Claude Code was used to assist in the reported simulations and production of the two Figures included here.

% \se

%%%%%%%%%%%%%%%%%%%%%%%%%%%%%%%%%%%%%%%%%%%%%%%%%%%%%%%%%%%%%%%
%**************************************************************
%%%%%%%%%%%%%%%%%%%%%%%%%%%%%%%%%%%%%%%%%%%%%%%%%%%%%%%%%%%%%%%

\bibliographystyle{plain}
\bibliography{Bibliography}

@article{ning2024assessment,
  title={An assessment of ensemble {K}alman filter and {A}zouani--{O}lson--{T}iti algorithms for data assimilation: A comparative study},
  author={Ning, Ning and Victor, Collin},
  journal={Communications in Computational Physics},
  volume={41},
  number={1},
  pages={213--236},
  year={2027},
  doi={10.4208/cicp.OA-2024-0167}
}

@article{titi2025inadequacy,
  title={On the inadequacy of nudging data assimilation algorithms for non-dissipative systems},
  author={Titi, Edriss S. and Victor, Collin},
  journal={Discrete and Continuous Dynamical Systems -- S},
  volume={26},
  pages={477--507},
  year={2026},
  doi={10.3934/dcdss.2025151}
}

@article{bessaih2025stochastic,
  title={Continuous data assimilation for {2D} stochastic {N}avier--{S}tokes equations},
  author={Bessaih, Hakima and Ferrario, Benedetta and Landoulsi, Oussama and Zanella, Margherita},
  journal={Journal of Mathematical Fluid Mechanics},
  volume={28},
  number={3},
  pages={62},
  year={2026},
  doi={10.1007/s00021-026-01043-4}
}

@article{hammoud2023continuous,
  title={Continuous and Discrete Data Assimilation with Noisy Observations for the {R}ayleigh--{B}{\'e}nard Convection: A Computational Study},
  author={Hammoud, Mohamad Abed El Rahman and Le Ma{\^i}tre, Olivier and Titi, Edriss S. and Hoteit, Ibrahim and Knio, Omar},
  journal={Computational Geosciences},
  volume={27},
  number={1},
  pages={63--79},
  year={2023},
  doi={10.1007/s10596-022-10180-4},
  url={https://doi.org/10.1007/s10596-022-10180-4},
  publisher={Springer}
}

@article{hammoud2024data,
  title={Data assimilation in chaotic systems using deep reinforcement learning},
  author={Hammoud, Mohamad Abed El Rahman and Raboudi, Naila and Titi, Edriss S. and Knio, Omar and Hoteit, Ibrahim},
  journal={Journal of Advances in Modeling Earth Systems},
  volume={16},
  number={8},
  pages={e2023MS004178},
  year={2024},
  publisher={Wiley Online Library},
  doi={10.1029/2023MS004178}
}

@article{balmforth1997master,
  title={Master-slave synchronization and the {L}orenz equations},
  author={Balmforth, Neil and Tresser, Charles and Worfolk, Patrick and Wu, Chai Wah},
  journal={Chaos: An Interdisciplinary Journal of Nonlinear Science},
  volume={7},
  number={3},
  pages={392--394},
  year={1997},
  publisher={American Institute of Physics},
  doi={10.1063/1.166212}
}

@article{desamsetti2022enhanced,
  title={Enhanced simulation of the {I}ndian summer monsoon rainfall using regional climate modeling and continuous data assimilation},
  author={Desamsetti, Srinivas and Dasari, Hari Prasad and Langodan, Sabique and Viswanadhapalli, Yesubabu and Attada, Raju and Luong, Thang M. and Knio, Omar and Titi, Edriss S. and Hoteit, Ibrahim},
  journal={Frontiers in Climate},
  volume={4},
  pages={817076},
  year={2022},
  publisher={Frontiers Media SA},
  doi={10.3389/fclim.2022.817076}
}

@article{hammoud2025downscaling,
  title={Downscaling using {CDAnet} under observational and model noise: the {R}ayleigh--{B}{\'e}nard convection paradigm},
  author={Hammoud, Mohamad Abed El Rahman and Titi, Edriss S. and Hoteit, Ibrahim and Knio, Omar},
  journal={Computational Geosciences},
  volume={29},
  number={1},
  pages={10},
  year={2025},
  publisher={Springer},
  doi={10.1007/s10596-024-10337-3}
}

@article{CDA_2_noise_2015,
   title={Continuous data assimilation with stochastically noisy data},
   volume={28},
   ISSN={1361-6544},
   url={http://dx.doi.org/10.1088/0951-7715/28/3/729},
   DOI={10.1088/0951-7715/28/3/729},
   number={3},
   journal={Nonlinearity},
   publisher={IOP Publishing},
   author={Bessaih, Hakima and Olson, Eric and Titi, Edriss S.},
   year={2015},
   month=feb, pages={729--753} }

@article{franz2022dynamic,
  title={The Bleeps, the Sweeps, and the Creeps: Convergence Rates for Dynamic Observer Patterns via Data Assimilation for the {2D} {N}avier--{S}tokes Equations},
  author={Franz, Trenton and Larios, Adam and Victor, Collin},
  journal={Computer Methods in Applied Mechanics and Engineering},
  volume={392},
  pages={114673},
  year={2022},
  doi={10.1016/j.cma.2022.114673},
  url={https://doi.org/10.1016/j.cma.2022.114673},
  publisher={Elsevier}
}

@article{biswas2021three,
  title={Continuous Data Assimilation for the Three-Dimensional {N}avier--{S}tokes Equations},
  author={Biswas, Animikh and Price, Randy},
  journal={SIAM Journal on Mathematical Analysis},
  volume={53},
  number={6},
  pages={6697--6723},
  year={2021},
  doi={10.1137/20M1323229},
  url={https://doi.org/10.1137/20M1323229},
  publisher={SIAM}
}

@article{larios2024ocean,
  title={Application of Continuous Data Assimilation in High-Resolution Ocean Modeling},
  author={Larios, Adam and Petersen, Mark R. and Victor, Collin},
  journal={Communications in Computational Physics},
  volume={35},
  number={5},
  pages={1418--1444},
  year={2024},
  doi={10.4208/cicp.OA-2023-0208},
  url={https://doi.org/10.4208/cicp.OA-2023-0208}
}

@article{larios2020approximate,
  title={Approximate continuous data assimilation of the {2D} {N}avier--{S}tokes equations via the {V}oigt-regularization with observable data},
  author={Larios, Adam and Pei, Yuan},
  journal={Evolution Equations and Control Theory},
  volume={9},
  number={3},
  pages={733--751},
  year={2020},
  doi={10.3934/eect.2020031}
}

@article{cibik2025adaptive,
  title={Adaptive Parameter Selection in Nudging Based Data Assimilation},
  author={{\c{C}}{\i}b{\i}k, Aytekin and Fang, Rui and Layton, William and Siddiqua, Farjana},
  journal={Computer Methods in Applied Mechanics and Engineering},
  volume={433},
  pages={117526},
  year={2025},
  doi={10.1016/j.cma.2024.117526},
  url={https://doi.org/10.1016/j.cma.2024.117526},
  publisher={Elsevier}
}

@article{martinez2025relaxation,
  title={Relaxation-based schemes for on-the-fly parameter estimation in dissipative dynamical systems},
  author={Martinez, Vincent R. and Murri, Jacob and Whitehead, Jared P.},
  journal={Inverse Problems},
  volume={41},
  number={5},
  pages={055001},
  year={2025},
  doi={10.1088/1361-6420/adc766},
  url={https://doi.org/10.1088/1361-6420/adc766},
  publisher={IOP Publishing}
}

@misc{ahmad2026dataassimilationaugmented,
  title={A Data-Assimilation-Augmented Optimization Framework for Parameter Estimation in Dynamical Systems},
  author={Ahmad, Muhammad Jalil and Biswas, Animikh and Hoffman, Kathleen},
  year={2026},
  eprint={2606.29291},
  archivePrefix={arXiv},
  primaryClass={math.DS},
  doi={10.48550/arXiv.2606.29291},
  url={https://arxiv.org/abs/2606.29291},
  howpublished={arXiv preprint arXiv:2606.29291}
}

@misc{wang2026comparing,
  title={Comparing Deterministic and Stochastic Parameter Recovery Algorithms Applied to Chaotic Systems},
  author={Wang, Ashley and Carlson, Elizabeth and Hoffmann, Franca},
  year={2026},
  eprint={2606.18568},
  archivePrefix={arXiv},
  primaryClass={nlin.CD},
  doi={10.48550/arXiv.2606.18568},
  url={https://arxiv.org/abs/2606.18568},
  howpublished={arXiv preprint arXiv:2606.18568}
}

@misc{li2026continuoussurrogate,
  title={Continuous Data Assimilation with Learned Surrogate Dynamics},
  author={Li, Wenwen and Sanz-Alonso, Daniel},
  year={2026},
  eprint={2606.00480},
  archivePrefix={arXiv},
  primaryClass={math.DS},
  doi={10.48550/arXiv.2606.00480},
  url={https://arxiv.org/abs/2606.00480},
  howpublished={arXiv preprint arXiv:2606.00480}
}

@article{sun2020optimization,
  title={Optimization for deep learning: An overview},
  author={Sun, Ruo-Yu},
  journal={Journal of the Operations Research Society of China},
  volume={8},
  number={2},
  pages={249--294},
  year={2020},
  publisher={Springer},
  doi={10.1007/s40305-020-00309-6}
}

@article{lorenz1963deterministic,
  title={Deterministic nonperiodic flow},
  author={Lorenz, Edward N.},
  journal={Journal of the Atmospheric Sciences},
  volume={20},
  number={2},
  pages={130--141},
  year={1963},
  doi={10.1175/1520-0469(1963)020<0130:DNF>2.0.CO;2}
}

@article{brunton2016discovering,
  title={Discovering governing equations from data by sparse identification of nonlinear dynamical systems},
  author={Brunton, Steven L. and Proctor, Joshua L. and Kutz, J. Nathan},
  journal={Proceedings of the National Academy of Sciences},
  volume={113},
  number={15},
  pages={3932--3937},
  year={2016},
  publisher={National Academy of Sciences},
  doi={10.1073/pnas.1517384113}
}

@book{sparrow2012lorenz,
  title={The {L}orenz Equations: Bifurcations, Chaos, and Strange Attractors},
  author={Sparrow, Colin},
  series={Applied Mathematical Sciences},
  volume={41},
  year={1982},
  publisher={Springer-Verlag},
  address={New York},
  doi={10.1007/978-1-4612-5767-7}
}

@article{newey2025model,
  title={Model discovery on the fly using continuous data assimilation},
  author={Newey, Joshua and Whitehead, Jared P. and Carlson, Elizabeth},
  journal={Journal of Computational Physics},
  volume={537},
  pages={114121},
  year={2025},
  publisher={Elsevier},
  doi={10.1016/j.jcp.2025.114121}
}

@article{foias2016discrete,
  title={A discrete data assimilation scheme for the solutions of the two-dimensional {N}avier--{S}tokes equations and their statistics},
  author={Foias, Ciprian and Mondaini, Cecilia F. and Titi, Edriss S.},
  journal={SIAM Journal on Applied Dynamical Systems},
  volume={15},
  number={4},
  pages={2109--2142},
  year={2016},
  publisher={SIAM},
  doi={10.1137/16M1076526}
}

@article{pachev2022concurrentmultiparameterlearningdemonstrated,
  title={Concurrent multiparameter learning demonstrated on the {K}uramoto--{S}ivashinsky equation},
  author={Pachev, Benjamin and Whitehead, Jared P. and McQuarrie, Shane A.},
  journal={SIAM Journal on Scientific Computing},
  volume={44},
  number={5},
  pages={A2974--A2990},
  year={2022},
  doi={10.1137/21M1426109}
}

@article{farhat2020data,
  title={Data assimilation in large {P}randtl {R}ayleigh--{B}énard convection from thermal measurements},
  author={Farhat, Aseel and Glatt-Holtz, Nathan E. and Martinez, Vincent R. and McQuarrie, Shane A. and Whitehead, Jared P.},
  journal={SIAM Journal on Applied Dynamical Systems},
  volume={19},
  number={1},
  pages={510--540},
  year={2020},
  publisher={SIAM},
  doi={10.1137/19M1248327}
}

@article{ AOT_2014,
Author = {Azouani, Abderrahim and Olson, Eric and Titi, Edriss S.},
Title = {Continuous Data Assimilation Using General Interpolant Observables},
Journal = {Journal of Nonlinear Science},
Year = {2014},
Volume = {24},
Number = {2},
Pages = {277-304},
month = apr,
DOI = {10.1007/s00332-013-9189-y},
ISSN = {0938-8974},
EISSN = {1432-1467},
ORCID-Numbers = {Titi, Edriss/0000-0002-5004-1746},
Unique-ID = {WOS:000333528200004},
}

@article{jolly2019continuous,
  title={Continuous data assimilation with blurred-in-time measurements of the surface quasi-geostrophic equation},
  author={Jolly, Michael S. and Martinez, Vincent R. and Olson, Eric J. and Titi, Edriss S.},
  journal={Chinese Annals of Mathematics, Series B},
  volume={40},
  number={5},
  pages={721--764},
  year={2019},
  publisher={Springer},
  doi={10.1007/s11401-019-0158-0}
}

@article{yushutin2025continuous,
  title={Continuous data assimilation for problems with limited regularity using non-interpolant observables},
  author={Yushutin, Vladimir},
  journal={Journal of Numerical Mathematics},
  year={2026},
  note={Online first, doi:10.1515/jnma-2025-0046},
  doi={10.1515/jnma-2025-0046}
}

@article{biswas2021mesh,
  title={Mesh-free interpolant observables for continuous data assimilation},
  author={Biswas, Animikh and Brown, Kenneth R. and Martinez, Vincent R.},
  journal={Annals of Applied Mathematics},
  volume={38},
  number={3},
  pages={296--355},
  year={2022},
  doi={10.4208/aam.OA-2022-0006}
}

@article{mangan2017model,
  title={Model selection for dynamical systems via sparse regression and information criteria},
  author={Mangan, Niall M. and Kutz, J. Nathan and Brunton, Steven L. and Proctor, Joshua L.},
  journal={Proceedings of the Royal Society A: Mathematical, Physical and Engineering Sciences},
  volume={473},
  number={2204},
  pages={20170009},
  year={2017},
  publisher={The Royal Society},
  doi={10.1098/rspa.2017.0009}
}

@article{messenger2021weak,
  title={Weak {SINDy}: {G}alerkin-based data-driven model selection},
  author={Messenger, Daniel A. and Bortz, David M.},
  journal={Multiscale Modeling \& Simulation},
  volume={19},
  number={3},
  pages={1474--1497},
  year={2021},
  publisher={SIAM},
  doi={10.1137/20M1343166}
}

@article{viknesh2026adam,
  title={{ADAM-SINDy}: An efficient optimization framework for parameterized nonlinear dynamical system identification},
  author={Viknesh, Siva and Tatari, Younes and Christenson, Chase and Arzani, Amirhossein},
  journal={Physical Review Research},
  volume={8},
  number={1},
  pages={013040},
  year={2026},
  publisher={APS},
  doi={10.1103/dwkk-5g2h}
}

@article{wu2019numerical,
  title={Numerical aspects for approximating governing equations using data},
  author={Wu, Kailiang and Xiu, Dongbin},
  journal={Journal of Computational Physics},
  volume={384},
  pages={200--221},
  year={2019},
  publisher={Elsevier},
  doi={10.1016/j.jcp.2019.01.030}
}

@inproceedings{cranmer2020discovering,
  title={Discovering symbolic models from deep learning with inductive biases},
  author={Cranmer, Miles and Sanchez Gonzalez, Alvaro and Battaglia, Peter and Xu, Rui and Cranmer, Kyle and Spergel, David and Ho, Shirley},
  booktitle={Advances in Neural Information Processing Systems},
  volume={33},
  pages={17429--17442},
  year={2020}
}

@article{farhat2018assimilation,
  title={Assimilation of nearly turbulent {R}ayleigh--{B}{\'e}nard flow through vorticity or local circulation measurements: a computational study},
  author={Farhat, Aseel and Johnston, Hans and Jolly, Michael and Titi, Edriss S.},
  journal={Journal of Scientific Computing},
  volume={77},
  number={3},
  pages={1519--1533},
  year={2018},
  publisher={Springer},
  doi={10.1007/s10915-018-0686-x}
}

@article{biswas2021data,
  title={Data Assimilation for the {N}avier--{S}tokes Equations Using Local Observables},
  author={Biswas, Animikh and Bradshaw, Zachary and Jolly, Michael S.},
  journal={SIAM Journal on Applied Dynamical Systems},
  volume={20},
  number={4},
  pages={2174--2203},
  year={2021},
  publisher={SIAM},
  doi={10.1137/20M136058X}
}

@article{martinez2025reconstruction,
  title={On the reconstruction of unknown driving forces from low-mode observations in the {2D} {N}avier--{S}tokes equations},
  author={Martinez, Vincent R.},
  journal={Proceedings of the Royal Society of Edinburgh Section A: Mathematics},
  volume={155},
  number={6},
  pages={2171--2194},
  year={2025},
  publisher={Royal Society of Edinburgh Scotland Foundation},
  doi={10.1017/prm.2024.31}
}

@mastersthesis{schill2026asymptotic,
  title={Asymptotic Sensitivity and Adjoint Methods for Equation Discovery with Open-Source Implementation},
  author={Schill, Nathan},
  school={Brigham Young University},
  year={2026},
  note={BYU ScholarsArchive Theses and Dissertations 11194},
  url={https://scholarsarchive.byu.edu/etd/11194}
}

@misc{garrigos2023handbook,
  title={Handbook of convergence theorems for (stochastic) gradient methods},
  author={Garrigos, Guillaume and Gower, Robert M.},
  howpublished={arXiv preprint arXiv:2301.11235},
  year={2023},
  eprint={2301.11235},
  archivePrefix={arXiv}
}

@article{farhat2024identifying,
  title={Identifying the body force from partial observations of a two-dimensional incompressible velocity field},
  author={Farhat, Aseel and Larios, Adam and Martinez, Vincent R. and Whitehead, Jared P.},
  journal={Physical Review Fluids},
  volume={9},
  number={5},
  pages={054602},
  year={2024},
  publisher={APS},
  doi={10.1103/PhysRevFluids.9.054602}
}

@article{ CHL_2020,
Author = {Carlson, Elizabeth and Hudson, Joshua and Larios, Adam},
Title = {Parameter Recovery for the 2 Dimensional {N}avier--{S}tokes Equations via Continuous Data Assimilation},
Journal = {SIAM Journal on Scientific Computing},
Year = {2020},
Volume = {42},
Number = {1},
Pages = {A250-A270},
DOI = {10.1137/19M1248583},
ISSN = {1064-8275},
EISSN = {1095-7197},
ORCID-Numbers = {Hudson, Joshua/0000-0003-2764-4538
   Carlson, Elizabeth/0000-0002-5773-9311
   Larios, Adam/0000-0001-8066-1439},
Unique-ID = {WOS:000551241700025},
}

@article{Martinez_2022,
   title={Convergence analysis of a viscosity parameter recovery algorithm for the {2D} {N}avier--{S}tokes equations},
   volume={35},
   ISSN={1361-6544},
   url={http://dx.doi.org/10.1088/1361-6544/ac5362},
   DOI={10.1088/1361-6544/ac5362},
   number={5},
   journal={Nonlinearity},
   publisher={IOP Publishing},
   author={Martinez, Vincent R.},
   year={2022},
   month=apr, pages={2241--2287} }

@misc{alcalde2026data,
  title={Data-driven discovery of polynomial {ODEs} with provably bounded solutions},
  author={Alcalde, Albert and Fantuzzi, Giovanni},
  howpublished={arXiv preprint arXiv:2604.26933},
  year={2026},
  eprint={2604.26933},
  archivePrefix={arXiv}
}

@article{Lorenz_2021,
  title={Dynamically learning the parameters of a chaotic system using partial observations},
  author={Carlson, Elizabeth and Hudson, Joshua and Larios, Adam and Martinez, Vincent R. and Ng, Eunice and Whitehead, Jared P.},
  journal={Discrete and Continuous Dynamical Systems},
  volume={42},
  number={8},
  pages={3809--3839},
  year={2022},
  doi={10.3934/dcds.2022033}
}

@book{asch2016data,
  title={Data Assimilation: Methods, Algorithms, and Applications},
  author={Asch, Mark and Bocquet, Marc and Nodet, Ma{\"e}lle},
  year={2016},
  publisher={SIAM},
  doi={10.1137/1.9781611974546}
}

@article{alizadeh2022advances,
  title={Advances and challenges in climate modeling},
  author={Alizadeh, Omid},
  journal={Climatic Change},
  volume={170},
  number={1--2},
  pages={18},
  year={2022},
  publisher={Springer},
  doi={10.1007/s10584-021-03298-4}
}

@article{beck2019deep,
  title={Deep neural networks for data-driven {LES} closure models},
  author={Beck, Andrea and Flad, David and Munz, Claus-Dieter},
  journal={Journal of Computational Physics},
  volume={398},
  pages={108910},
  year={2019},
  publisher={Elsevier},
  doi={10.1016/j.jcp.2019.108910}
}

@article{grooms2023backscatter,
  title={Backscatter in energetically-constrained {L}eith parameterizations},
  author={Grooms, Ian},
  journal={Ocean Modelling},
  volume={186},
  pages={102265},
  year={2023},
  publisher={Elsevier},
  doi={10.1016/j.ocemod.2023.102265}
}

@article{guan2023learning,
  title={Learning physics-constrained subgrid-scale closures in the small-data regime for stable and accurate {LES}},
  author={Guan, Yifei and Subel, Adam and Chattopadhyay, Ashesh and Hassanzadeh, Pedram},
  journal={Physica D: Nonlinear Phenomena},
  volume={443},
  pages={133568},
  year={2023},
  publisher={Elsevier},
  doi={10.1016/j.physd.2022.133568}
}

@article{carrassi2018data,
  title={Data assimilation in the geosciences: An overview of methods, issues, and perspectives},
  author={Carrassi, Alberto and Bocquet, Marc and Bertino, Laurent and Evensen, Geir},
  journal={Wiley Interdisciplinary Reviews: Climate Change},
  volume={9},
  number={5},
  pages={e535},
  year={2018},
  publisher={Wiley Online Library},
  doi={10.1002/wcc.535}
}

@book{law2015data,
  title={Data Assimilation: A Mathematical Introduction},
  author={Law, Kody and Stuart, Andrew and Zygalakis, Konstantinos},
  series={Texts in Applied Mathematics},
  volume={62},
  year={2015},
  publisher={Springer},
  address={Cham},
  doi={10.1007/978-3-319-20325-6}
}

\end{document}